\documentclass{article}
\usepackage{amsthm, amsfonts, amssymb, latexsym, color, tikz}
\usepackage{hyperref}
\usepackage{lipsum}

\title{The Elekes-Szab\'{o} Problem and the Uniformity Conjecture}
\newcommand{\E}{\mathsf{E}}

\newcommand{\Q}{\mathbb{Q}}
\newcommand{\R}{\mathbb{R}}
\usepackage{amsmath}
\usepackage[a4paper, total={6in, 8in}]{geometry}
\newcommand{\cC}{\mathcal{C}}

\author{Mehdi Makhul, Oliver Roche-Newton, Sophie Stevens and Audie Warren}

\newtheorem*{corollary*}{Corollary 5}
\newtheorem{definition}{Definition}
\newtheorem{theorem}{Theorem}
\newtheorem{corollary}{Corollary}
\newtheorem{conjecture}{Conjecture}

\newcommand{\Addresses}{{
  \bigskip
  \footnotesize

  M. Makhul, \textsc{Fakult\"{a}t f\"{u}r Mathematik Universit\"{a}t Wien, Austria}\par\nopagebreak
  \textit{E-mail address}, \texttt{mmakhul@risc.jku.at}

  \medskip

  O. Roche-Newton, \textsc{Johann Radon Institute for Computational and Applied Mathematics (RICAM), Linz, Austria}\par\nopagebreak
  \textit{E-mail address}, \texttt{o.rochenewton@gmail.com}

  \medskip

   S. Stevens, \textsc{Johann Radon Institute for Computational and Applied Mathematics (RICAM), Linz, Austria}\par\nopagebreak
  \textit{E-mail address}, \texttt{sophie.stevens@ricam.oeaw.ac.at}
  
  \medskip

   A. Warren, \textsc{Johann Radon Institute for Computational and Applied Mathematics (RICAM), Linz, Austria}\par\nopagebreak
  \textit{E-mail address}, \texttt{audie.warren@oeaw.ac.at}

}}

\begin{document}
 \maketitle

 \begin{abstract}
    In this paper we give a conditional improvement to the Elekes-Szab\'{o} problem over the rationals, assuming the Uniformity Conjecture. Our main result states that for $F\in \mathbb{Q}[x,y,z]$ belonging to a particular family of polynomials, and any finite sets $A, B, C \subset \mathbb Q$ with $|A|=|B|=|C|=n$, we have
    \[
    |Z(F) \cap (A\times B \times C)| \ll n^{2-\frac{1}{s}}.
    \]
    The value of the integer $s$ is dependent on the polynomial $F$, but is always bounded by $s \leq 5$, and so even in the worst applicable case this gives a quantitative improvement on a bound of Raz, Sharir and de Zeeuw \cite{RazSharirdeZeeuw}.
    
    We give several applications to problems in discrete geometry and arithmetic combinatorics. For instance, for any set $P \subset \mathbb Q^2$ and any two points $p_1,p_2 \in \mathbb Q^2$, we prove that at least one of the $p_i$ satisfies the bound
    \[
    | \{ \| p_i - p \| : p \in P \}| \gg |P|^{3/5},
    \]
    where $\| \cdot \|$ denotes Euclidean distance. This gives a conditional improvement to a result of Sharir and Solymosi \cite{SS}.
 \end{abstract}
 \section{Introduction}

A recurring theme in arithmetic combinatorics is the principle that ``additive and multiplicative structure cannot coexist". This has been quantified in many ways, but perhaps the most general approach is the study of \emph{expanders}\footnote{The precise meaning of the term \textit{expander} varies in the literature, although all definitions capture that basic idea that these polynomials give many different outputs.}:
functions $f(x_1,...,x_k)$ such that the cardinality of the image set, restricted to finite sets $A_i$ is large with respect to the size of the $A_i$'s, for any choice of sets $A_i$. In the simplest balanced case when the sets $A_i$ have the same cardinality $n$, expanders are functions which guarantee that 
\[|f(A_1,...,A_k)| \gg n^{1+c}\,,\]
where $f(A_1,...,A_k) :=  \left\{ f(a_1,...,a_k) : a_i \in A_i \right\}$ and $c$ is some positive absolute constant.

Throughout this paper, the notation
 $X\gg Y$, $Y \ll X,$ $X=\Omega(Y)$, and $Y=O(X)$ are all equivalent and mean that $X\geq cY$ for some absolute constant $c>0$. $X \gg_a Y$ means that the implied constant is no longer absolute, but depends only on $a$.

Expander functions have two main questions of interest: ($i$) classifying which functions $f$ are expanders; ($ii$) determining the rate of expansion. Applying the principle of the impossibility of the coexistence of additive and multiplicative structure (termed the \emph{sum-product phenomenon}), one should expect that as long as $f$ combines multiplication and addition in some non-trivial way, then $f$ is an expander. 

Elekes and R\'{o}nyai \cite{ElekesRonyai} gave a classification of two variable polynomials $f$ which are not expanders, by showing\footnote{In fact, the lower bound claimed in \cite{ElekesRonyai} was of the form $|f(A,B)| \gg \omega (n)$, but an inspection of the proof shows that it gives $|f(A,B)| \gg n^{1+c}$.} that there is some fixed constant $c>0$ such that either

\begin{itemize}
    \item $\forall A,B \subset \mathbb R$ with $|A|=|B|=n$ we have $|f(A,B)| \gg n^{1+c}$, or
    \item $f$ takes the form $f(x,y) = g(h(x) + k(y))$ or $f(x,y) = g(h(x)k(y))$ for some univariate polynomials $g,h,k$.
\end{itemize}
Any $f$ that is of one of the previous forms is called \emph{degenerate}. For non-degenerate $f$ and finite sets $A,B \subseteq \R$, finding lower bounds for $|f(A,B)|$ is termed the Elekes-R\'{o}nyai problem. The current best result in this direction is due to Raz, Sharir and Solymosi \cite{RazSharirSolymosi}. 
\begin{theorem}[Raz - Sharir - Solymosi; balanced case]\label{thm:RSS elekes ronyai}
Let $f \in \R[x,y]$ be a polynomial of degree $d$. If $f$ is not degenerate, then for all finite $A, B \subseteq \R$ with $|A| = |B| = n$, we have
    $$|f(A,B)| \gg_d n^{4/3}.$$
\end{theorem}

This is actually a special case of a more general phenomenon; notice that the surface ${F(x,y,z) = 0}$ given by $F(x,y,z) := f(x,y) - z$ has $|A||B|$ intersection points with the Cartesian product \sloppy{${A \times B \times f(A,B)}$}, so that we have
\begin{equation} \label{triviallower}
|A||B| \leq |Z(F) \cap A \times B \times f(A,B)|,
\end{equation}
where $Z(F)$ denotes the zero set of $F$ (such zero sets shall always be taken over the real numbers). Finding upper bounds for the number of intersections of the zero set of a polynomial $F\in \mathbb{F}[x,y,z]$ and a Cartesian product $A \times B \times C$, under certain non-degeneracy conditions, is called the Elekes-Szab\'{o} problem, see \cite{ElekesSzabo}. In light of the observation \eqref{triviallower}, such results translate well into Elekes-R\'{o}nyai type statements. The best result at present is due to Raz, Sharir and de Zeeuw \cite{RazSharirdeZeeuw}:
\begin{theorem}[Raz - Sharir - de Zeeuw; balanced case over $\mathbb{R}$] \label{11/6ES}
Let $F \in \R[x,y,z]$ be an irreducible (over $\mathbb{R}$) polynomial of degree $d$, such that $Z(F)$ has dimension $2$.
Then one of the following is true.
\begin{enumerate}
    \item \label{item:rsz upper}For all $A,B,C\subseteq \mathbb{R}$ of size $|A|=|B|=|C| = N$ we have
    \begin{equation}\label{eq:RSZ11/6}
    |Z(F) \cap (A \times B \times C)| = O_d( N^{11/6})\,;
    \end{equation}
    \item \label{item:exceptional}There exists a one-dimensional sub-variety $Z_0 \subseteq Z(F)$ such that for all $v\in Z(F)\setminus Z_0$, there are open intervals $I_1,I_2,I_3\subseteq \R$ and injective real-analytic functions $\phi_i:I_i\rightarrow \R$ with real-analytic inverses ($i = 1,2,3$) so that: for $ v \in I_1\times I_2\times I_3$ and for any $(x,y,z)\in I_1\times I_2\times I_3$, we have
    \[ (x, y, z) \in Z(F) \text{ if and only if } \phi_1(x) + \phi_2(y) + \phi_3(z) = 0\,.\]
\end{enumerate}
\end{theorem}

The question of finding the correct exponent in the upper bound of \eqref{eq:RSZ11/6} remains open. The best lower bound comes from a simple construction in \cite{MRNWdZ} of the non-degenerate polynomial $F(x,y,z)=(x-y)^2+x-z$ 
and sets $A,B$ and $C$ of cardinality $N$ with $|Z(F) \cap (A \times B \times C)| = \Omega( N^{3/2})$.

Although we state here the balanced case when $|A|=|B|=|C|$ over $\mathbb{R}$, Raz, Sharir, and de Zeeuw actually prove a stronger, unbalanced case over $\mathbb{C}$ -- we refer the reader to \cite[Theorem 1.2]{RazSharirdeZeeuw} for the details of this result. The exceptional form in Theorem \ref{11/6ES} was qualitatively discovered by Elekes and Szab\'o \cite[Theorem 3]{ElekesSzabo} and they provide a detailed account of the structure of the varieties $Z(F)$ corresponding to the exceptional form in both the real and complex cases, see \cite[Example 1]{ElekesSzabo}. Some simple examples of polynomials that are degenerate in the sense of Theorem \ref{11/6ES} are $F(x,y,z)=x+y+z$, $F(x,y,z)=x+y^2+z^2$ and $F(x,y,z)=xyz$.

In the unpublished article \cite{Solymosi}, Solymosi gives a much more specific formulation of degeneracy over $\mathbb{Q}$ for the Elekes-R\'onyai problem, namely that if $F(x,y,z)$ is not an expander over the rationals, then $F(x,y,z)$ has one of three very specific forms (see \cite[Theorems 1 and 2]{Solymosi} for the details of these forms).  This naturally translates into a more specific definition of degeneracy in the two variable situation of Theorem~\ref{thm:RSS elekes ronyai}.

The theme of this note is to demonstrate the link between the Elekes-Szab\'o problem and the algebraic geometric \emph{Uniformity Conjecture}.

\subsection{The Uniformity Conjecture and Hyperelliptic Curves}
In a landmark paper, Faltings \cite{Faltings} showed that the number of rational points on any curve defined over $\mathbb{Q}$ of genus at least 2 is finite; this number however may depend on the curve. In this paper we make use of the Uniformity Conjecture, see \cite{CaporasoHarrisMazur}, which asserts the existence of a uniform bound $B_g = B_g(\mathbb{Q})$ for the number of rational points on any smooth curve of genus $g\geq 2$ over $\mathbb{Q}$. This conjecture stands over any number field $K$ (with $B_g(K)$ substituted for $B_g(\mathbb{Q})$); we only make use of the formulation over $\Q$.

\begin{conjecture}[Uniformity Conjecture over $\mathbb{Q}$] \label{uniformityconjecture}
For every integer $g\geq 2$ there exists an integer $B_g$ such that any smooth curve defined over $\mathbb{Q}$ has at most $B_g$ rational points.
\end{conjecture}

It is proven in \cite{CaporasoHarrisMazur} that Conjecture \ref{uniformityconjecture} would follow from the Bombieri-Lang conjecture (also called the \emph{weak-Lang conjecture}), which is a deep conjecture in arithmetic geometry concerning the set of rational points on varieties of general type. Conjecture \ref{uniformityconjecture} has been used before to prove conditional results in combinatorics, particularly concerning square numbers\footnote{Throughout this paper a set of square numbers should be understood to mean a set of rational squares. A \emph{rational $k$th power} is a number $a^k$ such that $a \in \Q$.}. Cilleruelo and Granville \cite{CillerueloGranville} used Conjecture \ref{uniformityconjecture} to prove that for any set $A \subseteq \Q$ of square numbers we have $|A+A| \gg |A|^{4/3}$. For context, it is currently not known unconditionally that $|A+A| \gg |A|^{1+c}$ for some $c >0$. In fact Cilleruelo and Granville proved something stronger -- that the number of solutions to $a+b = c+d$ in any set of rational squares $A$ is at most $O(|A|^{8/3})$. It was conjectured by Chang \cite{Chang} that the number of such solutions is bounded by $|A|^{2 + \epsilon}$ for all $\epsilon > 0$. Conditional results have recently been given by Shkredov and Solymosi \cite{ShkredovSolymosi}, who improved on the aforementioned results of \cite{CillerueloGranville} and gave an optimal bound on the fourth order additive energy of a set of squares. 

Sets of rational squares are natural objects of interest, and their relevance to the Uniformity Conjecture is demonstrated via the choice of curve. A \emph{hyperelliptic curve} is a curve of the form $y^2 = f(x)$ with certain constraints on $f$. Suitably defined (see below), a hyperelliptic curve has genus $g\geq 2$, and so by Conjecture~\ref{uniformityconjecture} has at most $B_g$ rational points.

\begin{definition}\label{def:hyperelliptic}
A hyperelliptic curve $\mathcal{C}$ over $\mathbb{C}$ is the set of points $(x,y)\in \mathbb{C}^2$ satisfying $y^2 = f(x)$ for $f\in \mathbb{C}[x]$ of degree $d\geq 5$ so that $\gcd(f(x),f'(x)) = 1$ (i.e. $f$ has no repeated root in $\mathbb{C}$).
The genus of $\mathcal{C}$ is $g = \lfloor (d-1)/2\rfloor$.
\end{definition}

Note that the condition that $f$ has no repeated roots is equivalent to $\cC$ being non-singular over $\mathbb{C}$. Moreover, since $f$ has no repeated roots, it is in particular square free, and so the curve $\mathcal{C}$ is irreducible over $\mathbb{C}$. Thus Definition~\ref{def:hyperelliptic} coincides with the classical definition of a hyperelliptic curve (over $\mathbb{C}$): a non-singular projective irreducible curve over $\mathbb{C}$ of genus $g\geq 2$ with an affine equation of the form $y^2 = f(x)$. In this paper we restrict to $f(x) \in \mathbb{Q}[x]$. For a more thorough treatment of hyperelliptic curves we refer to Galbraith \cite[Chapter 10]{Galbraith}.

The Uniformity Conjecture is often applied to combinatorial problems via hyperelliptic curves; this is done by relating the object to be counted (e.g. the number of edges in a certain graph) to rational points on some hyperelliptic curves. For instance, to bound the fourth order additive energy of a set of squares, Shkredov and Solymosi \cite{ShkredovSolymosi} used the genus 2 hyperelliptic curve $y^2 = (x^2 +\alpha)(x^2 + \beta)(x^2 + \gamma)$ for non-zero $\alpha,\beta,\gamma \in \mathbb{Z}$.

The Uniformity Conjecture has been used to derive new information about sum-product problems, particularly sum-product problems considered on sparse graphs. For a bipartite graph $G$ taken over $A \times B$ for some finite sets $A,B \subseteq \Q$, we define the sum and product sets of $A$ and $B$ \emph{along} $G$ to be
$$A+_G B = \{ a + b : (a,b) \in E(G) \}, \quad A\cdot_G B = \{ a  b : (a,b) \in E(G) \}.$$ 
It was conjectured by Erd\H{o}s and Szemer\'{e}di that for any bipartite graph $G$ on $A \times A$ with $A$ a set of integers satisfying $|A|= N$ with $N$ sufficiently large, and $|E(G)| > N^{1+c}$ for $0<c\leq 1$ , we have
$$|A +_G A| + |A \cdot_G A| \gg |E(G)|^{1 - \epsilon}$$
for all $\epsilon >0$. This conjecture was refuted in \cite{ARS}, but the question of finding the best possible sum-product bounds on sparse graphs remains open. An elementary argument gives a lower bound for this problem of 
\begin{equation} \label{easy}
|A +_G A| + |A \cdot_G A| \gg |N|^{1/2}\,.
\end{equation}

The case when $G$ is a matching has attracted particular attention.  Alon, Angel, Benjamini, and Lubetzky \cite{Alon} used the Uniformity Conjecture to give an improvement to \eqref{easy}. Further progress was obtained by Shkredov and Solymosi \cite{ShkredovSolymosi}, still under the assumption of the Uniformity Conjecture. They focus on the case when $G$ is a matching, but their proof implicitly gives the following more general bound for sums and products along graphs: given $A,B \subset \mathbb Q$ and $G$ a bipartite graph on vertex set $A \times B$,
\begin{equation} \label{spgraphs}
   |A +_G B| + |A \cdot_G B| \gg |G|^{3/5} .
\end{equation}

Unconditionally, the best lower bound in this direction is due to Chang \cite{Chang} who proved a lower bound for a related problem: $ |A+_G A| + |A -_G A| + |A\cdot_G A| \gg N^{1/2}\log^{1/48}N$. Improved bounds for denser graphs can be obtained using the Szemer\'{e}di-Trotter Theorem; see for example \cite[Theorem 10]{ARS}.

\section{Main Results}

In this paper, we give conditional results that quantitatively improve the Elekes-Szab\'o problem in particular cases: we show that certain families of polynomials $F\in \mathbb{Q}[x,y,z]$ admit a stronger upper bound than that given in Theorem \ref{11/6ES} for finite sets $A,B,C\subseteq \mathbb{Q}$. For some polynomials we attain the optimal upper bound $|Z(F)\times A\times B\times C|= O_d(|A|+|B|+|C|)$
. We then present a variety of applications.
\subsection{An Elekes-Szab\'o type theorem}
Our main result is the following.

 \begin{theorem} \label{thm:main}
 Assume the Uniformity Conjecture is true, and let $F(x,y,z)$ be a polynomial of degree $d$ of the form $q(x,y,z)^2 - p(x,y)$, with $q(x,y,z) \in \Q[x,y,z]$ and $p(x,y) \in \Q[x,y]$, with $\frac{\partial F}{\partial z} \neq 0$. Let $d_p := \deg_y(p) \geq 1$, and define $s$ to be the smallest integer such that $s d_p \geq 5$. Let $A,B,C \subseteq \Q$ be any finite sets, and assume that for any $a \in A$, $p(a,y)$ has no repeated roots. Then we have
 \begin{equation} \label{long}|Z(F) \cap A \times B \times C| \ll_{s,d} |A||B|^{1-1/s}  +|A|^{1-1/s}|B|+|M_{A,p}|^{1/s}|B| + L_F|C|
 \end{equation}

 where
 \begin{align*}
 M_{A,p}:=  \{(a_1,\dots, a_s) \in A^s : a_i &\text{ are distinct and there exists } i \neq j \text{ such that } p(a_i,y) \text{ and } p(a_j,y)
 \\&\text{ have a common root in } \mathbb{C} \}
 \end{align*}
 and 
$$L_F := |\{ (a,b) \in A \times B : F(a,b,z) \equiv 0 \}|.$$
 \end{theorem}

Although the statement of Theorem \ref{thm:main} looks somewhat complicated and unwieldy, we will give several applications using concrete polynomials $F(x,y,z)$ which correspond to natural problems in discrete geometry and combinatorics. The reader should keep in mind that the important terms in \eqref{long} are the first two, and that in all of our applications we can perform calculations to ensure that the other terms can be absorbed by these. This means that in the worst applicable case, when $\deg_y(p) = 1$, a balanced application of Theorem \ref{thm:main} (setting $|A|=|B|=|C|=N$) gives an upper bound of $N^{9/5}$, which improves upon the bound of $N^{11/6}$ given by Theorem \ref{11/6ES}. As $d_p$ increases the result improves, and when $d_p \geq 5$, Theorem \ref{thm:main} gives the optimal bound of $O_d(|A| + |B| + |C|)$. In the case where $F(x,y,z)$ is an irreducible polynomial, we have the bound $L_F\leq d^2$ -- see Raz, Sharir and de Zeeuw \cite[Lemma 2.1]{RazSharirdeZeeuw}.

An interesting feature of Theorem \ref{thm:main} is that there exist polynomials $F$ for which Theorem~\ref{thm:main} gives a non-trivial bound on the intersection of $Z(F)$ with a Cartesian product, but that are degenerate in the sense of Theorem \ref{11/6ES}. A simple example is $F(x,y,z)=x+y+z^2$. This highlights the different nature of the Elekes-Szab\'{o} problem when it is considered over $\mathbb Q$ rather than $\mathbb R$. We remark that we believe it is possible to give similar results for polynomials of the form $F(x,y,z) = q(x,y,z)^k - p(x,y)$ for $k\geq 3$, although we do not pursue this here. 

\subsection{Applications: Discrete Geometry}
\label{Sec:discretegeometryapplications}
The Elekes-Szab\'o problem was introduced in conjunction with an algebraic approach to questions of combinatorial geometry: at a very high level, it is often the case that questions in discrete geometry, particularly questions related to distances, can be reduced to the study of the zeroes of an appropriately chosen polynomial that lie in a Cartesian product. Thus, progress on the Elekes-Szab\'o problem leads to progress towards certain combinatorial geometric problems.

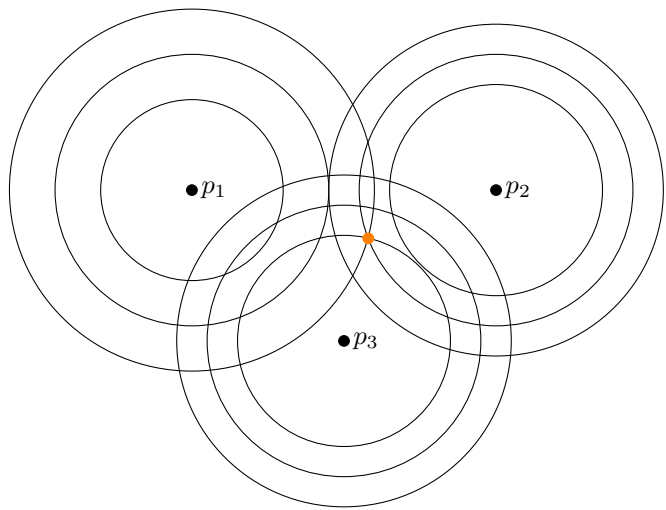
\begin{figure}[h!]
    \centering
       \begin{tikzpicture}
\draw[color=black](-2,0) circle (2.4);
\draw[color=black](-2,0) circle (1.8);
\draw[color=black](-2,0) circle (1.2);

\draw[color=black](2,0) circle (2.2);
\draw[color=black](2,0) circle (1.8);
\draw[color=black](2,0) circle (1.4);

\draw[color=black](0,-2) circle (2.2);
\draw[color=black](0,-2) circle (1.8);
\draw[color=black](0,-2) circle (1.4);

\filldraw [black] (-2,0) circle (2pt) node[anchor=west] {$p_1$};
\filldraw [black] (2,0) circle (2pt) node[anchor=west] {$p_2$};
\filldraw [black] (0,-2) circle (2pt) node[anchor=west] {$p_3$};

\filldraw [orange] (0.32,-0.64) circle (2pt) node[anchor=west] {};
\end{tikzpicture}
    \caption{Three families of concentric circles. A triple intersection point is highlighted.}
    \label{fig:my_label}
\end{figure}

Consider the following problem: suppose we have three families of $n$ concentric circles $\mathcal C_1, \mathcal C_2, \mathcal C_3$, centred at $p_1, p_2$ and $p_3$ respectively. Suppose also that the three points $p_i \in \mathbb R^2$ are not collinear\footnote{A construction of Elekes \cite{Elekes} gives a system of three sets of concentric circles with collinear centres determining $\Omega(n^2)$ triple intersections, and so this additional restriction is necessary.}. Let $T(\mathcal C_1, \mathcal C_2, \mathcal C_3)$ denote the set of triple intersection points of the three families of circles:
\[
T(\mathcal C_1, \mathcal C_2, \mathcal C_3):=\{(C_1,C_2,C_3) \in \mathcal C_1 \times \mathcal C_2 \times \mathcal C_3 : C_1 \cap C_2 \cap C_3 \neq \emptyset \}.
\]

The question of how large this set can be was first raised by Erd\H{o}s, Lov\'{a}sz and Vesztergombi in \cite{ELV}. Elekes and Szab\'{o} \cite{ElekesSzabo} showed that it can be reduced to a question of bounding the intersection of the zero set of a non-degenerate polynomial with a Cartesian product. The main result of their paper, a bound for what we now call the Elekes-Szab\'{o} problem, led to a non-trivial upper bound $|T(\mathcal C_1,\mathcal C_2, \mathcal C_3)| \ll n^{2-c}$. The work of Raz, Sharir and de Zeeuw \cite{RazSharirdeZeeuw} gives the improved upper bound
\begin{equation} \label{11over6}
|T(\mathcal C_1,\mathcal C_2, \mathcal C_3)| \ll n^{11/6}.
\end{equation}
We improve to this estimate, conditional on the Uniformity Conjecture, together with the additional assumptions that the circles have rational centres and that the squares of the radii are rational numbers.

\begin{corollary} \label{cor:triples}
Assume the Uniformity Conjecture and suppose we have three families of concentric circles $\mathcal C_1, \mathcal C_2, \mathcal C_3$, centred at non-collinear rational points $p_1, p_2$ and $p_3$ respectively. Define $R_i$ to be the set of radii of the circles in $\mathcal C_i$ and suppose that $R_i^2 \subset \mathbb Q$. Then 
\[
|T(\mathcal C_1,\mathcal C_2, \mathcal C_3)| \ll |\mathcal C_1||\mathcal C_2|^{2/3}+ |\mathcal C_1|^{2/3}|\mathcal C_2|.
\] \end{corollary}
In particular, this gives the bound
\[
|T(\mathcal C_1,\mathcal C_2, \mathcal C_3)| \ll n^{5/3}
\]
for the balanced case when each of the family of pencils has cardinality $n$, thus improving \eqref{11over6}.

The set $\mathcal{C}_3$ does not appear in the upper bound in Corollary ~\ref{cor:triples}, and this omission can be interpreted as follows: suppose we have three families  of concentric circles that determine many (in the sense of Corollary~\ref{cor:triples}) triple points. Then at least one of the following must happen (i) the centres of the circles are collinear (ii) one centre has an irrational coordinate (iii) a positive proportion of the squares of radii in one family of circles is in $\mathbb{R}\setminus\mathbb{Q}$.

We can also prove the following more natural and quantitatively stronger result, the difference being that we must have rational radii.
\begin{corollary} \label{cor:triples2}

Assume the Uniformity Conjecture and suppose we have three families of concentric circles $\mathcal C_1, \mathcal C_2, \mathcal C_3$, centred at non-collinear rational points $p_1, p_2$ and $p_3$ respectively. Define $R_i$ to be the set of radii of the circles in $\mathcal C_i$ and suppose that $R_i \subset \mathbb Q$. Then
\[
|T(\mathcal C_1,\mathcal C_2, \mathcal C_3)| \ll |\mathcal C_1||\mathcal C_2|^{1/2}+ |\mathcal C_1|^{1/2}|\mathcal C_2|.
\]
\end{corollary}

The reason we prove the slightly more unnatural Corollary \ref{cor:triples} is its application to the following distinct distances problem, which can be viewed as a strengthened form of the Erd\H{o}s pinned distance problem.

 Let $P \subset \mathbb R^2$ and let $p_1,p_2,p_3 \in \mathbb R^2$ be three non-collinear points. Let
\[
D(p_i,P):=\{ \|p_i -p \| : p \in P \}
\]
denote the set of distances between $p_i$ and elements of $P$. Can we prove that at least one of the sets $D(p_i,P)$ is large? The current state-of-the-art bound, due to Sharir and Solymosi \cite{SS}, states that
\[
\max \{ |D(p_1,P)|,  |D(p_2,P)|,  |D(p_3,P)| \} \gg |P|^{6/11}.
\]
Under the assumption of the Uniformity Conjecture, we can give an analogous bound for distances between \emph{two} rational points and a set of rational points.

\begin{corollary} \label{distances}
Assume the Uniformity Conjecture and let $P \subset \mathbb Q^2$ be finite. Then for any two points $p_1,p_2 \in \mathbb Q^2$ 
\begin{equation} \label{general}
|D(p_1,P)|^3|D(p_2,P)|^2 +  |D(p_1,P)|^2|D(p_2,P)|^3 \gg |P|^{3}.
\end{equation}

In particular 
\[
\max \{ |D(p_1,P)|,  |D(p_2,P)| \} \gg |P|^{3/5}.
\]
\end{corollary}

As well as giving a quantitative improvement to the main result of \cite{SS}, Corollary \ref{distances} exposes a difference in behaviour between rational and real points; there exists a set of points $P \subseteq \R^2$ and two points $p_1,p_2 \in \R^2$ such that there are only $O(|P|^{1/2})$ distances from $p_i$ to $P$. Indeed, this is seen by taking two sets of $N$ concentric circles, letting $P$ be the set of two-rich points defined by these circles, and taking $p_1,p_2$ to be the centres of the circles.  We have $|P| = 2N^2$, and $|D(p_1,P)| = |D(p_2,P)| = N$. Corollary \ref{distances} shows that, assuming the Uniformity Conjecture, such constructions cannot occur in $\Q^2$.

We can use Corollary \ref{distances} to give bounds on the number of rational points defined by the intersection of two sets of concentric circles.
\begin{corollary} \label{rationalintersections}
Assume the Uniformity Conjecture, and let $\mathcal C_1$ and $\mathcal C_2$ be two sets of concentric circles with $|\mathcal C_1| = |\mathcal C_2| = N$, centred at points $p_1, p_2 \in \Q^2$ respectively. Let $\mathcal C_1 \cap \mathcal C_2$ denote the set of points lying on one circle from each family. Then the number of rational points in $\mathcal C_1 \cap \mathcal C_2$ is $O(N^{5/3})$.
\end{corollary}

For further applications of Uniformity-type conjectures to problems of combinatorial geometry, see, for instance, Shaffaf \cite{shaffaf} and Tao \cite{tao}, who independently used the Bombieri-Lang conjecture 
to give a negative answer to the Erd\H{o}s-Ulam problem \cite{erdosulam}: is there a dense (with respect to the Euclidean topology) subset $S$ of the real plane so that the distances between any two points in $S$ is a rational number? See also van Luijk \cite{vanluijk} for an application of the Bombieri-Lang conjecture to counting perfect cuboids. To our knowledge, \cite{vanluijk} was the first work to apply the Bombieri-Lang conjecture to discrete geometry.

\subsection{Applications: Arithmetic Combinatorics}

Theorem \ref{thm:main} can be naturally applied to give an Elekes-R\'{o}nyai result for a family of polynomials in $\mathbb{Q}[x,y]$. Similar to the situation with Theorem \ref{thm:main}, an interesting feature of this family is that it contains polynomials that expand over the rationals, but \emph{not} over the reals. In addition, the (conditional) rate of expansion is strong: for certain $f(x,y)\in \mathbb{Q}[x,y]$ of degree $d$ we can show optimal expansion, that $|\{f(a,b)\colon a\in A, b\in B\}|\gg_d |A||B|$ for any $A,B\subseteq \mathbb{Q}$. We prove the following.

\begin{corollary} \label{Cor:expanders}
Assume the Uniformity Conjecture, and let $h(x)$ and $g(x,y)$ be rational polynomials of degree at most $d$, such that $\frac{\partial g}{\partial y},\frac{dh}{dx}  \neq 0$, and for any $t \in \{ h(a) + g(a,b)^2 : a\in A, b \in B \}$, $t - h(x)$ has no repeated roots. Define $F(x,y) = h(x) + g(x,y)^2$. Then for any finite sets $A,B \subseteq \Q$ with $|A| \leq |B|$, we have
$$\left| F(A,B)\right| \gg_d |B||A|^{1/s}$$
where $s$ is the smallest integer such that $s \deg(h) \geq 5$.
\end{corollary}
In the preprint \cite{Solymosi}, Solymosi observed that the Uniformity Conjecture can be used towards refining the qualitative aspect of the Elekes-R\'onyai problem. That is, refining the families of polynomials in $\mathbb{Q}[x,y]$ that expand. The focus in Corollary \ref{Cor:expanders} is on the quantitative aspect, dealing with the rate of expansion.

Corollary \ref{Cor:expanders} 
can be used to give results concerning the sum set of a set of squares. In particular, it can be used to reprove the bound
\begin{equation} \label{sums}
    |A+A| \gg |A|^{4/3}
\end{equation}
for any set $A$ of rational squares, a result which first appeared in \cite{CillerueloGranville} and was recently improved slightly in \cite{ShkredovSolymosi}. These results are related to a conjecture of Rudin \cite{Rudin}, which states that the number of squares in an arithmetic progression $A$ is $O(|A|^{1/2})$. Current progress for this conjecture is given by Bombieri and Zannier \cite{BombieriZannier} who proved that $A$ can contain at most $|A|^{3/5 + \epsilon}$ squares. Shkredov and Solymosi used the Uniformity Conjecture to extend this result, allowing $A$ to be a generalised arithmetic progression, see \cite[Theorem 1]{ShkredovSolymosi}. With a simple application of \eqref{sums}, we are able to generalise this further to sets with small sum set, with a good dependence on the doubling of $A$.

\begin{corollary}\label{Cor:Rudin}
Assume the Uniformity Conjecture, and let $A \subseteq \Q$ be a set such that $|A+A| \leq K|A|$ for some $K$. Let $B$ be the set of square numbers in $A$. Then we have
$$|B| \ll (K|A|)^{3/4}.$$
\end{corollary}

Theorem \ref{thm:main} also allows us to give information on the number of three term arithmetic progressions ($3$APs) in sets of squares. It is well known that no four square numbers can be in arithmetic progression; therefore any set of rational squares contains no (non-trivial) $4$APs. 
One might expect that a set containing no $4$APs cannot contain too many $3$APs, however this is false: a result of Fox and Pohoata \cite{PohoataFox} implies that a set with no $4$APs may still contain more than $|A|^{2-o(1)}$ $3$APs. We prove the following.
\begin{corollary}\label{3APs}
Assume the Uniformity Conjecture, let $A\subseteq \Q$ be a set of rational squares.
Then the number $AP_3(A)$ of three term arithmetic progressions in $A$ satisfies
$$AP_3(A) \ll |A|^{5/3}.$$
\end{corollary}
This improves upon a bound of $O(|A|^{7/4})$ (also conditional on the Uniformity Conjecture) which follows immediately from an application of H\"{o}lder's inequality along with the optimal energy\footnote{See the forthcoming definition in \eqref{energy}} bound $\E_4(A) \ll |A|^{4}$ given in \cite{ShkredovSolymosi}. Indeed, we have 
$$AP_3(A) = \sum_{x \in 2A}r_{A+A}(x) \leq |A|^{3/4}\E_4(A)^{1/4} \ll|A|^{7/4}.$$
An unconditional non-trivial bound $AP_3(A)=o(|A|^2)$ for sets of squares follows from the work of Balog and Szemer\'{e}di \cite{BS} (see also \cite[Theorem 1.7]{Elekes2}), and a more precise unconditional quantitative bound can be recovered from the upper bound in \cite[Theorem 1.2]{PohoataFox}. 
It should be noted that there exists a set $A$ of rational squares containing $\Omega(|A|\log|A|)$ $3$APs. Indeed, the first $n$ square numbers contain $\Omega(n\log n)$ Pythagorean triples (see \cite{Nowak}). Each solution to $a^2 + b^2 = c^2$ with $a^2,b^2,c^2$ with $a \geq b$ in the first $n$ squares gives a triple $(a-b)^2,c^2,(a+b)^2$, which is a three term arithmetic progression contained in the first $2n$ squares. Therefore the set of the first $2n$ square numbers contains $\Omega(n \log n)$ $3$APs. If $A$ is a set of $k$th powers for $k \geq 3$, stronger unconditional results are known for $AP_3(A)$: from a result of Darmon and Merel \cite{DarmonMerel}, one can deduce that $AP_3(A)\ll |A|$. 

As a further application of Conjecture \ref{uniformityconjecture}, we consider the additive energy of sets $A$ and $B$ where $A$ is a set of rational squares, and $B$ is a set of rational $k$th powers. The \textit{additive energy} of $A$ and $B$, denoted $\E(A,B)$, is the number of solutions to the equation $a-b=a'-b'$ such that $(a,a',b,b') \in A^2 \times B^2$. This can also be written as
\[
\E(A,B)= \sum_{x \in A-B} r_{A-B}^2(x)
\]
where $r_{A-B}(x):=|\{(a,b) \in A \times B : a-b= x\}|$. Given a real number $l >0$, we can generalise this to the $l$th order additive energy of $A$ and $B$, which is defined as
\begin{equation} \label{energy}
\E_l(A,B)= \sum_{x \in A-B} r_{A-B}^l(x).
\end{equation}

We write $\E_l(A,A) = \E_l(A)$. Shkredov and Solymosi \cite{ShkredovSolymosi}, under the Uniformity Conjecture, were able to show that $\E(A)= o(|A|^{8/3})$ for $A\subseteq \mathbb{Q}$ a set of squares and $\E(A)=o(|A|^{5/2})$ for $A$ a set of cubes. If $A$ is a set of $k$th powers for $k\geq 4$, they remark that the same technique yields an optimal energy bound $\E(A) = O(|A|^2)$. We extend their result to the asymmetric case $\E(A,B)$, where $A$ is a set of squares and $B$ is a set of $k$th powers for $k \geq 5$. 
\begin{corollary} \label{energybounds}
Let $A$ be a set of rational squares, and $B$ be a set of rational $k$th powers, for $k \geq 5$. Let $l>0$ be a real number. Then assuming the Uniformity Conjecture, we have
$$\E_l(A,B) \ll_k |A-B| + |A \cap B|^l.$$
\end{corollary}
An important variant of the sum-product problem is that of products of additive shifts: given sets $A$ and $B$ over some field $\mathbb{F}$, is it true that at least one of $AB$ or $(A+\alpha)(B+\beta)$ must be large, for some additive shifts $\alpha$ and $\beta$? See, for instance, \cite{GS}, \cite{HRNZ} and \cite{Shkredov} for work on questions of this type. We consider this problem when the product set and shifted product set are constructed relative to a graph, obtaining the following result.
\begin{corollary} \label{cor:sp}
Assume the Uniformity Conjecture, let $A,B \subset  \mathbb Q \setminus \{0\} $, $\alpha, \beta \in \mathbb Q \setminus \{ 0 \}$, and $G \subset A \times B$. Then
\[
|A\cdot_G B| + |(A+\alpha) \cdot_G (B+\beta)| \gg |E(G)|^{3/5}.
\]
\end{corollary}

Corollary \ref{cor:sp} is an analogue of \eqref{spgraphs} for products and shifted products, and is the first non-trivial bound for this problem over sparse graphs. Similar results over dense graphs can be found in \cite{RNW}.

\section{Proof of Theorem \ref{thm:main}}

We begin by removing some bad elements from $A$, to be dealt with separately. We write the polynomial $p(x,y)$ as a univariate polynomial in $(\Q[x])[y]$ (recall that $d_p := \deg_y(p)$):
$$p(x)(y) := \sum_{i=1}^{d_p} r_i(x)y^i.$$
We define  $R = \{ a \in A : r_{d_p}(a) = 0 \}$. Elements of $R$ have the undesirable effect of reducing the degree of $p(x)(y)$. We have
\begin{equation}\label{eq:split}|Z(F) \cap A \times B \times C| =  |Z(F) \cap (A \setminus R) \times B \times C| + |Z(F) \cap R \times B \times C|
\end{equation}
and we begin by bounding the first term. Let $ A' := A \setminus R$, and note that since $r_{d_p}(x)$ is a univariate polynomial of degree at most $d$, we have $|R| \leq d$ and so $|A'|\geq |A|- d$.

Define a bipartite graph $G = (V,E)$ on $A' \times B$ as follows: a pair $(a,b) \in A' \times B$ forms an edge if there exists some $c \in C$ such that $F(a,b,c) = 0$. To each edge $e = (a,b) \in E$ we assign a weight $w(e)$, defined as 
$$w((a,b)) = \left| \left\{ c \in C : F(a,b,c)=0 \right\} \right|$$
so that
$$|Z(F) \cap A' \times B \times C| = \sum_{e \in E} w(e).$$

Note that for $(a,b,c) \in A' \times B \times C$ with $F(a,b,c) = 0$, we have that either $H_{a,b}(z) := F(a,b,z)$ is identically zero, or there are at most $d$ values of $c \in C$ with $H_{a,b}(c) = 0$. Therefore, for each edge $e = (a,b) \in E$ we either have $w(e) \leq d$, or $w(e) = |C|$. Recall the notation $$L_F:=|\{(a,b)\in A\times B\colon F(a,b,z)\equiv 0\}|.$$

We then have
\begin{equation}\label{bound1} |Z(F) \cap A' \times B \times C| = \sum_{e \in E} w(e) \leq d|E| + L_F|C|.\end{equation}

We now aim to find a bound on $|E|$. Let $s$ be the smallest integer such that $sd_p\geq 5$. Then by H\"{o}lder's inequality
\begin{equation} \label{edges}
|E|^s= \left( \sum_{b \in B} \deg(b) \right )^s \leq |B|^{s-1} \left ( \sum_{b \in B} \deg^s(b) \right)= |B|^{s-1} \sum_{a_1, \dots, a_s \in A'} |N_G(a_1) \cap \dots \cap N_G(a_s)|.
\end{equation}
The remaining task is to bound the quantity
\[
S:=\sum_{a_1, \dots, a_s \in A'} |N_G(a_1) \cap \dots \cap N_G(a_s)|.
\]
Split this sum into three parts by writing
\begin{equation} \label{decompose}
S=S_1+ S_2 +S_3,
\end{equation}
where
\begin{itemize}
    \item $S_1$ counts the contributions to $S$ for which $a_1,\dots,a_s$ are all distinct and for all $i \neq j$ there does not exist any solution $y_0 \in \mathbb C$ to the equation $p(a_i,y_0)=p(a_j,y_0)=0$,
    \item $S_2$ counts the contributions to $S$ for which $a_1,\dots,a_s$ are all distinct and there exists $i \neq j$ and some $y_0 \in \mathbb C$ such that $p(a_i,y_0)=p(a_j,y_0)=0$,
    \item $S_3$ counts the contributions to $S$ for which $a_1,\dots,a_s$ are not pairwise distinct.
\end{itemize}

First, let us bound $S_1$. Fix distinct $a_1,\dots,a_s \in A'$ such that there does not exist any solution to the equation $p(a_i,y)=p(a_j,y)=0$. We define the auxiliary polynomial
\begin{equation} \label{aux1}
Y^2 = \prod_{i=1}^s p(a_i,X),
\end{equation}
which has degree at least five. Indeed, by the definition of $A'$, each $p(a_i,X)$ has degree $d_p$, and we also have the condition that $s d_p \geq 5$. Furthermore, by our assumption that we are considering contributions to $S_1$, as well as the hypothesis that the roots of $p(a_i,X)$ are distinct, it follows that the roots of
\[
\prod_{i=1}^s p(a_i,X)
\]
are all distinct. Therefore equation \eqref{aux1} gives a hyperelliptic curve of genus $g:= \lfloor \frac{s d_p-1}{2} \rfloor \geq 2$, and the Uniformity Conjecture asserts that this curve has at most $B_g$ rational points. However, if $ b \in N_G(a_1) \cap \dots \cap N_G(a_s)$ then we have
\[
F(a_i,b,c_i)=0
\]
for some $c_1,\dots,c_s \in C$. By construction we see that
$$\left( b, \prod_{i=1}^s q(a_i,b,c_i) \right)$$
is a rational solution to equation \eqref{aux1}. It therefore follows that $|N_G(a_1) \cap \dots \cap N_G(a_s)| \leq B_g$ and thus
\begin{equation} \label{S1bound}
S_1 \ll |A|^s.
\end{equation}

We now turn to the term $S_2$. By the definition of the set $M_{A,p}$,
\[
S_2= \sum_{(a_1,\dots,a_s) \in M_{A,p}} |N_G(a_1) \cap \dots \cap N_G(a_s)|.
\]
Therefore, the trivial bound 
\begin{equation} \label{trivial}
|N_G(a_1) \cap \dots \cap N_G(a_s)| \leq |B|
\end{equation}
implies that
\begin{equation} 
\label{S2bound}
S_2 \leq |M_{A,p}||B|.
\end{equation}

To bound $S_3$, we simply observe that the number of valid $s$-tuples which may contribute to this sum is at most $\binom{s}{2}|A|^{s-1}$. Therefore, another application of \eqref{trivial} implies that
\begin{equation} \label{S3bound}
S_3 \ll_s|A|^{s-1}|B|.
\end{equation}
Combining \eqref{decompose}, \eqref{S1bound}, \eqref{S2bound} and \eqref{S3bound} yields
\[
S \ll_s |A|^s + |M_{A,p}||B| + |A|^{s-1}|B|.
\]
It then follows from \eqref{edges} that
\[
|E| \ll_s |A||B|^{1- 1/s} + |M_{A,p}|^{1/s}|B| + |A|^{1-1/s}|B|
\] 
and so by equation \eqref{bound1} we have 
\begin{equation} \label{eq:endresult}
    |Z(F) \cap A' \times B \times C| \ll_{s,d} |A||B|^{1-1/s} +|M_{A,p}|^{1/s}|B|+ |B||A|^{1-1/s} + L_F|C|.
\end{equation}

We now deal with the second term in \eqref{eq:split}. Suppose we have a triple $(a,b,c) \in R \times B \times C$ such that $F(a,b,c) = 0$. The univariate polynomial $H_{a,b}(z) := F(a,b,z)$ is either identically zero, or has at most $d$ zeroes. There are at most $L_F$ pairs $(a,b)$ where the first case occurs, giving at most $L_F|C|$ zeroes. On the other hand, recalling that $|R| \leq d$, there are at most $d^2|B|$ zeroes  $(a,b,c)$ where $H_{a,b}(z)$ is not identically zero, since there are $d|B|$ pairs $(a,b) \in R \times B$, each giving at most $d$ zeroes. Putting this together, we have
\begin{equation} \label{2ndterm}
|Z(F) \cap R \times B \times C| \ll_d |B| + L_F|C|.
\end{equation}
Combining \eqref{2ndterm}, \eqref{eq:endresult} and \eqref{eq:split} completes the proof. \qedsymbol

\section{Proofs of the applications to discrete geometry}
In this section we prove Corollary \ref{cor:triples}, and use this in turn to derive Corollary \ref{distances}. We also indicate the changes to the proof of Corollary \ref{cor:triples} that can be made to prove Corollary \ref{cor:triples2}. 

\begin{proof}[Proof of Corollary \ref{cor:triples}]
After translating and dilating our original configuration, we may assume without loss of generality that
\[
p_1=(0,0),\,\,\, p_2=(1,a),\,\,\, p_3=(b,c).
\]
Our initial assumption that these three points are not collinear is now equivalent to the statement $ab \neq c$. Recall that $R_i$ is the set of the radii of the circles in $C_i$ and that $R_i^2 := \{r^2: r\in R_i\}$ is the set of the squares of the radii.

For a point $(u,v) \in \mathbb R^2$, treating $u$ and $v$ as variables, we define $X$, $Y$ and $Z$ to be the square of the distances  between the point $(u,v)$ and respectively $p_1,p_2$ and $p_3$. That is
\begin{align} \label{circlesystem}
X&=u^2+v^2 \nonumber
\\ Y&=(u-1)^2 + (v-a)^2 \nonumber
\\ Z &= (u-b)^2 + (v-c)^2.
\end{align}
We view this as a system of equations with 5 variables $(u,v,X,Y,Z)$. A triple point for the families of circles $\mathcal C_1$, $\mathcal C_2$ and $\mathcal C_3$ corresponds to a solution $(u,v,X,Y,Z) \in \mathbb R \times \mathbb R \times R_1^2 \times R_2^2 \times R_3^2$.

We can eliminate $u$ and $v$ from the system \eqref{circlesystem} to deduce that any triple intersection point gives us a solution to
\begin{equation} \label{defn}
F(X,Y,Z)=0, \,\,\, (X,Y,Z) \in R_1^2 \times R_2^2 \times R_3^2
\end{equation}
for a quadratic polynomial $F$. That is,
\begin{equation} \label{setup}
|T(\mathcal C_1,\mathcal C_2, \mathcal C_3)| = |Z(F) \cap R_1^2 \times R_2^2 \times R_3^2|.
\end{equation}
We use SAGE to help with calculating the explicit form of $F$, which is
\begin{align*} 
F(X,Y,Z)&=X^2(a^2-2ac+b^2-2b+c^2+1)  + X Y(2ac - 2b^2 +2 b - 2 c^{2} ) \\&+XZ(-2a^2+2ac+2b-2)+ Y^{2} (b^{2} +  c^{2}) + Y Z(-2 a c - 2 b)
\\&+ X( - 2 a^{3} c - 2  a^{2} b + 4  a^{2} c^{2} - 2  a b^{2} c + 8  a b c - 2  a c^{3} - 2  a c - 2  b^{3} +  4 b^{2} - 2  b c^{2} - 2  b) 
 \\&+ Y (-2a^{2} b^{2} - 2  a^{2} c^{2} + 2  a b^{2} c + 2  a c^{3} + 2  b^{3} - 2  b^{2} + 2  b c^{2} - 2  c^{2})
 \\&+ Z^{2}( a^{2} + 1) +  Z(2 a^{3} c - 2  a^{2} b^{2} + 2  a^{2} b - 2  a^{2} c^{2} + 2  a c - 2 b^{2} + 2  b - 2  c^{2})
\\& + a^{4} b^{2} + a^{4} c^{2} - 2 a^{3} b^{2} c - 2 a^{3} c^{3} + a^{2} b^{4} - 2 a^{2} b^{3} + 2 a^{2} b^{2} c^{2} + 2 a^{2} b^{2} - 2 a^{2} b c^{2} + a^{2} c^{4}  \\& + 2 a^{2} c^{2} - 2 a b^{2} c - 2 a c^{3} + b^{4} - 2 b^{3} + 2 b^{2} c^{2} + b^{2} - 2 b c^{2} + c^{4} + c^{2}.
\end{align*}

Let $F(X,Y,Z) = (a^2 + 1) G(X,Y,Z)$. Note that, as a univariate polynomial in $Z$ over the polynomial ring $\mathbb{Q}[X,Y]$, $G$ is now monic, and so $L_G = 0$. 

We can express $G(X,Y,Z)$ in the form
\[
G(X,Y,Z) = (q(X,Y,Z))^2 - p(X,Y)
\]

where 
\[
q(X,Y,Z) =Z + \frac{\left(-a^{2} + a c + b - 1\right) X + \left(-a c - b\right) Y - a^{2} b^{2} + a^{3} c - a^{2} c^{2} + a^{2} b - b^{2} + a c - c^{2} + b}{a^2+1}
\] 
and 

\[
p(X,Y) = -\left(\frac{ab-c}{a^2+1}\right)^2 \left( Y^{2} - 2 \left(a^2+1+X \right)Y + (a^{2} +1  - X )^{2}\right).
\]

Note that the coefficient of $Y^2$ in $p$ is $-\left(\frac{ab-c}{a^2+1}\right)^2\in \mathbb{Q}$. Since $ab\neq c$, we can conclude that $\deg_y(p) = 2$.

 For fixed $x$, $p(x,Y)$ is a quadratic equation with discriminant $\Delta= 16x(a^2+1)$. The roots of this polynomial are \[(a^2+1+x)\pm 2\sqrt{x(a^2+1)}.\]
In particular $p(x,Y)$ has a repeated root if and only if $x=0$. Since $0 \notin R_1^2$, the requirement that $p(x,Y)$ does not have any repeated roots for all $x \in R_1^2$ is satisfied.

The next part of the analysis is needed to give control of the size of $M_{R_1^2,p}$. Let $x_1$ and $x_2$ be two fixed values. We see that $p(x_1,Y)$ and $p(x_2,Y)$ have a common root if and only if one of the four variants of the equation $\sqrt{x_1} \pm \sqrt{x_2} = \pm 2\sqrt{a^2+1}$ holds. This implies that for each $r \in R_1^2$, there exist at most four possible choices of $r' \in R_1^2$ such that $p(r,Y)$ and $p(r',Y)$ share a common root. Thus, recalling that $s=3$, we have
\[
|M_{R_1^2,p}| \leq 4 \binom{3}{2}|R_1|^{2}.
\]
We can finally apply Theorem \ref{thm:main} with $s=3$, which gives
\begin{equation} \label{above}
|Z(G) \cap R_1^2 \times R_2^2 \times R_3^2| \ll |\mathcal C_1||\mathcal C_2|^{2/3}+|\mathcal C_1|^{2/3}|\mathcal C_2|.
\end{equation}
Since $Z(G)=Z(F)$, the proof is complete by combining \eqref{above} with \eqref{setup}.
\end{proof}

Corollary \ref{cor:triples2} is proved in much the same way, the differences being that in equation \eqref{circlesystem} we may replace $X$ with $X^2$, $Y$ with $Y^2$ and $Z$ with $Z^2$. We thus consider solutions to $q(X^2,Y^2,Z^2)^2 = p(X^2,Y^2)$. This has the effect of changing $\deg_Y(p((X,Y))$ to four, so that we may apply Theorem \ref{thm:main} with $s = 2$. With minor changes to the analysis of the common roots, we find that $|M_{R_1,p}| \ll |R_1|$, and we conclude that 
$$|Z(G) \cap R_1 \times R_2 \times R_3| \ll |\mathcal C_1||\mathcal C_2|^{1/2}+|\mathcal C_1|^{1/2}|\mathcal C_2|.$$

\begin{proof}[Proof of Corollary \ref{distances}]
Consider a fixed point set $P \subset \mathbb{Q}^2$ and two fixed points $p_1, p_2 \in \mathbb Q^2$. Choose $p_3 \in \mathbb Q^2$ to be any point which is not on the line through $p_1$ and $p_2$. Define $\mathcal C_i$ to be the set of all circles centred at $p_i$ which cover $P$. Note that $|\mathcal C_i|=|D(p_i,P)|$. Observe first that
\[
|P| \leq T(\mathcal C_1, \mathcal C_2, \mathcal C_3),
\]
since each point of $P$ is a triple intersection point for these three families of circles. On the other hand, since all of the points under consideration have rational coordinates, the sets $D(p_i,P)^2$ are contained in $\mathbb Q$. Therefore, Corollary \ref{cor:triples} can be applied to this system to obtain the upper bound
\[
T(\mathcal C_1, \mathcal C_2, \mathcal C_3) \ll |D(p_1,P)||D(p_2,P)|^{2/3}+|D(p_1,P)|^{2/3}|D(p_2,P)| .
\]
Comparing the upper and lower bounds completes the proof of \eqref{general}.
\end{proof}

In proving Corollary~\ref{distances}, we appeal to Corollary~\ref{cor:triples2}, which treats families of circles centred at \emph{three} points whereas Corollary~\ref{distances} treats families of circles centred at \emph{two} points. This may suggest to the reader that Corollary~\ref{cor:triples2} is unnecessary in the proof. This is indeed the case, and the same conclusion follows directly from Theorem~\ref{thm:main}. However, this then requires some somewhat lengthy calculations, essentially modifications of the proof of Corolary \ref{cor:triples}, which we avoid in this presentation of the proof.

We finally prove Corollary \ref{rationalintersections}.

\begin{proof}[Proof of Corollary \ref{rationalintersections}]
As in the statement of Corollary \ref{rationalintersections}, let $\mathcal{C}_1$ and $\mathcal{C}_2$ be two sets of $N$ concentric circles, centred at the rational points $p_1$ and $p_2$ respectively. Let $P \subset \mathcal{C}_1 \cap \mathcal{C}_2$ be the set of rational points lying on the set of points determined by the intersection of circles in $\mathcal C_1$ and circles in $\mathcal{C}_2$. By Corollary \ref{distances}, we have $$|D(p_1,P)|^{2/3}|D(p_2,P)| + |D(p_1,P)||D(p_2,P)|^{2/3} \gg |P|.$$
Since (trivially) $|D(p_1,P)|,|D(p_2,P)|\leq N$, we conclude that $|P| \ll N^{5/3}$.
\end{proof}
\section{Proofs of the applications to arithmetic combinatorics}

\subsection{Expanders}

In this subsection we prove Corollary \ref{Cor:expanders}. For convenience, we restate Corollary \ref{Cor:expanders}.
\begin{corollary*}
Let $h(x)$ and $g(x,y)$ be rational polynomials of degree at most $d$, such that $\frac{\partial g}{\partial y},\frac{d h}{d x}  \neq 0$. Let $A,B \subseteq \Q$ be finite sets with $2d+1 \leq |A| \leq |B|$, and suppose that $t - h(x)$ has no repeated roots for any $t \in \{ h(a) + g(a,b)^2 : a\in A, b \in B \}$. Then, assuming the Uniformity Conjecture, we have
$$\left|  \left\{h(a) + g(a,b)^2 : a \in A, b\in B  \right\}\right| \gg_d |B||A|^{1/s}$$
where $s$ is the smallest integer such that $s \deg(h) \geq 5$.
\end{corollary*}
\begin{proof}
Define 
$$E := \left\{h(a) + g(a,b)^2 : a \in A, b\in B  \right\}. $$
We shall apply Theorem \ref{thm:main} to the polynomial $F(u,v,w) := h(v) - u + g(v,w)^2$
to find a bound on $|Z(F) \cap E \times A \times B|.$

In our application we have $q(u,v,w) = g(v,w)$, and $p(u,v) = u - h(v)$. 

Note that $|M_{E,p}|=0$. Indeed, let 
$e,e' \in E$ and consider $p(e,v)$ and $p(e',v)$ as univariate polynomials in $v$. If $v_0$ is a common root of both, then $e-h(v_0)=0=e'-h(v_0)$, which implies that $e=e'$. By assumption, we also have that $p(e,v)= e - h(v)$ has no repeated roots for $e\in E$.

We claim that $L_F \leq d$. Indeed, by writing $g(v,w)$ as a univariate polynomial in $w$, we have
$$F(u,v,w) = \left( \sum_{i=0}^k g_i(v)w^{i} \right)^2 + h(v) - u$$
for some $k \leq d$. 
In particular, considering $F$ momentarily as a univariate polynomial in $w$, we see that the coefficient of its leading term is a univariate polynomial in $v$ -- this coefficient is $g_k(v)^2$. 
We have that $\deg(g_k) \leq d$. Suppose $(a,b) \in A \times B$ is such that $F(a,b,w) \equiv 0$. We see that $b$ is a zero of $g_k$, and so only $d$ values of $b$ are possible. 
Moreover, for a fixed $b$ there is at most one $a \in A$ with $F(a,b,w) \equiv 0$, since the constant term is linear in $u$, and must be chosen as $a = g_0(b)^2 + h(b)$. Thus $L_F \leq d$. 

Applying Theorem \ref{thm:main} we have
$$|Z(F) \cap E \times A \times B| \ll_{d,s} |E||A|^{1-1/s} + |A||E|^{1-1/s} + |B|. $$
We also have the lower bound $|Z(F) \cap E \times A \times B| \geq |A||B|$, since for any pair $(a,b) \in A \times B$, the corresponding element $e:= h(a) + g(a,b)^2 \in E$ gives a zero $F(e,a,b)$. This yields
\begin{equation}\label{graphbound}
|A||B| \ll_{d,s} |E||A|^{1 - 1/s} + |A||E|^{1-1/s}+ |B|.
\end{equation}
If the first term above dominates, we are done. If the second term dominates, by comparing first and second terms we find $|E| \ll |A|$, and we also have $|B|^{\frac{s}{s-1}} \ll |E|$, so that
$$|B|^{\frac{s}{s-1}} \ll |E| \ll |A| \leq |B|$$
and we find that $|B| \ll 1$, and there is nothing to prove.

If the third term dominates we have $|A| \ll 1$. For a fixed $a \in A$, we consider the univariate polynomial $h(a) + g(a,y)^2$, which has degree at most $2d$. Let $R:= \{ a\in A : h(a) + g(a,y)^2 \equiv 0\}$ and note that $|R| \leq 2d$. Let $a_0 \in A \setminus R $ (note that $A\setminus R\neq \emptyset$ since $|A| \geq 2d+1$). For any $c \in \R$, the number of solutions to $h(a_0) + g(a_0,y)^2 = c$ is at most $2d$, so we have
$$|\{ h(a_0) + g(a_0,b)^2 : b \in B \}| \geq \frac{|B|}{2d} \gg_d |B| \gg |A||B|,$$
which is stronger than needed.
\end{proof}
Corollary \ref{Cor:expanders} gives information about sums of $(2k)$th powers; setting $h(x) = x^{2k}$ and $g(x,y) = y^k$ we find that for all $A$, $B \subseteq \Q$,
$$|A^{2k} + B^{2k}| \gg |B||A|^{1/s}$$
with $s = \lceil \frac{5}{2k} \rceil$. Here, $A^k$ denotes the set $\{a^k : a \in A\}$. When $k \geq 3$ this gives the optimal bound of $|A||B|$, and for fourth powers we find a bound of $|B||A|^{1/2}$. In the case $k=2$, we find that for $A, B \subseteq \Q$ finite sets of rational squares, we have $|A+B| \gg |A||B|^{1/3}$. Similar bounds for this asymmetric sums of squares problem were obtained in \cite{CillerueloGranville} and \cite{ShkredovSolymosi}.
\subsection{Bounding squares in sets with small doubling}

Recall that Corollary \ref{Cor:Rudin} states that if $A \subset \mathbb Q$ satisfies $|A+A| \leq K|A|$ then $A$ contains $O(K^{3/4}|A|^{3/4})$ squares.

\begin{proof}[Proof of Corollary \ref{Cor:Rudin}]
The proof of Corollary \ref{Cor:Rudin} is very simple, and uses only the result of Cilleruelo and Granville that for a set of squares $B$, we have $|B+B| \gg |B|^{4/3}$. Let $A \subseteq \Q$ be a set satisfying $|A+A| \leq K|A|$, and let $B \subseteq A$ be the set of rational squares contained in $A$. Then we have
$$|B|^{4/3} \ll |B+B| \leq |A+A| \leq K |A|$$
and the result follows.
\end{proof}

Note that we can obtain a small improvement to Corollary \ref{Cor:Rudin} by using the improved lower bounds from \cite{ShkredovSolymosi} for $|A+A|$ when $A$ is a set of squares. Moreover, this argument shows that an optimal lower bound $|A+A| \geq |A|^{2-\epsilon}$ for sets of squares would imply an optimal upper bound $|A|^{\frac{1}{2}+\epsilon'}$ for the number of squares contained in an arithmetic progression $A$, even giving a generalisation to sets with small sum set.

\subsection{Additive configurations in sets of squares}
In this section we prove Corollary \ref{3APs}.

\begin{proof}[Proof of Corollary \ref{3APs}]

Recall that we wish to show that $AP_3(A)\ll |A|^{5/3}$ for $A$ a set of rational squares.
  Let $A \subseteq \Q$ be a finite set of rational squares, and let $B = A^{1/2}$ be the set of positive roots of elements of $A$. We shall apply Theorem \ref{thm:main} to the polynomial $F(x,y,z) = z^{2} - \frac{1}{2}(x^{2} + y^{2})$ and the set $B \times B \times B$. We may remove $0$ from $B$ if it is present, since there are at most $O(|B|) = O(|A|)$ solutions to $F(x,y,z) = 0$ with one of $x,y,z = 0$. We have $q(x,y,z) = z$, and $p(x,y) = \frac{1}{2}(x^{2} + y^{2})$. We see that we have $\frac{\partial F}{\partial z} \neq 0$. Furthermore, since $0 \notin B$, we know that for any $b \in B$, the polynomial $p(b,y)$ does not have repeating roots. This shows that an application of Theorem \ref{thm:main} is valid here. 

  For distinct $b,b' \in B$, the univariate polynomials $p(b,y)$ and $p(b',y)$ have distinct roots, so that $|M_{A,p}| = 0$. Moreover, since $F$ is monic in $z$ we have $L_F = 0$.
  Applying Theorem \ref{thm:main} to $B \times B \times B$ gives
$$|Z(F) \cap B \times B \times B| \ll |B|^{5/3}$$
 Note that a zero $(a,b,c) \in B\times B \times B$ of $F$ corresponds to a solution to the equation $a^{2} + b^{2} = 2c^{2}$, which in turn gives a solution to $\alpha + \beta = 2\gamma$, with $\alpha, \beta, \gamma \in A$. Such solutions correspond to three term arithmetic progressions in $A$, so that we have
$$AP_3(A) = |Z(F) \cap B \times B \times B|  \ll |B|^{5/3} = |A|^{5/3}$$
as required.
\end{proof}

Theorem \ref{thm:main} can also be used to reprove other results concerning additive structure of squares. We can give an alternative proof for the following result, which is implicit from the work of Alon et al. \cite{Alon}. 
\begin{corollary} \label{decomposition2} Assume the Uniformity Conjecture. If $C \subset \mathbb Q$ is a finite set of squares and $C=A+B$ for some $A,B \subset \mathbb Q$ then either $|A|=O(1)$ or $|B|=O(1)$.
\end{corollary}
This is a consequence of the following result, which is essentially \cite[Theorem 14]{ShkredovSolymosi}. 
\begin{corollary}\label{decomposition}
Assume the Uniformity Conjecture, and let $C \subset \Q$ be a set of squares. Then for any finite non-empty $A,B \subseteq \Q$, we have
$$| \{ (a,b) \in A \times B : a + b \in C \} | \ll |A||B|^{4/5} + |B| |A|^{4/5}.$$
\end{corollary}

\begin{proof}[Proof of Corollary \ref{decomposition}]
 Let $C \subset \Q$ be a set of squares, and suppose that $A,B \subset \Q$. We shall apply Theorem \ref{thm:main} to the polynomial $F(x,y,z) = z^2 - x - y$, and the sets $A$, $B$, and $C^{1/2}$, where we take only the positive roots. We have $q(x,y,z) = z$, and $p(x,y) = x+y$. We see that $p(a,y)$ has no repeated roots, and furthermore that $|M_{A,p}| = L_F = 0$. Applying Theorem \ref{thm:main} with $s = 5$ then gives 
 $$|Z(F) \cap A \times B \times C^{1/2}| \ll |A||B|^{4/5} + |B||A|^{4/5}.$$
 Now consider the set 
 $$S = \{ (a,b) \in A \times B : a+b \in C\}.$$
 Every pair $(a,b)$ in $S$ gives an element $(a,b,(a+b)^{1/2})$ in $ A \times B \times C^{1/2} \cap Z(F)$. Therefore we have 
 $$|S| = |Z(F) \cap A \times B \times C^{1/2}| \ll |A||B|^{4/5} + |B||A|^{4/5},$$
giving the result.
\end{proof} 
Corollary \ref{decomposition} now follows immediately. Indeed let $C$ be a finite set of squares and suppose that $C=A+B$ for some $A,B \subset \mathbb Q$. Then by Corollary \ref{decomposition},
\[
|A||B|=|\{(a,b) \in A \times B : a+b \in C \}| \ll |A||B|^{4/5}+|A|^{4/5}|B|.
\]
This inequality implies that either $|A| \ll 1$ or $|B| \ll 1$ must hold.

As mentioned above, Corollary \ref{decomposition} is implicit in the work of Alon et al. \cite{Alon}. Their alternative proof is as follows: Suppose that $C=A+B$ and assume that $|A| \geq 5$. For any five distinct elements $a_1,\dots,a_5 \in A$, consider the auxiliary polynomial
\[
Y^2=(X+a_1)(X+a_2)(X+a_3)(X+a_4)(X+a_5).
\]
By the Uniformity Conjecture, this has $O(1)$ rational solutions. But each $b \in B$ gives the rational solution $(b, \sqrt {(b+a_1)(b+a_2)(b+a_3)(b+a_4)(b+a_5)})$, and so $|B| \ll 1$.

\subsection{Bounds on Energies}
In this section we prove Corollary \ref{energybounds}, bounding $\E_l(A,B)$ for $A$ a set of rational squares, and $B$ a set of rational $k$th powers, for $k \geq 5$. 

\begin{proof}

We will actually prove the following stronger $l_{\infty}$ bound for the representation function $r_{A-B}$: 
\begin{equation} \label{linfty}
\forall m \neq 0, r_{A-B}(m) \ll_k 1.
\end{equation}
Since $a$ is a square and $b$ is a $k$th power, for each representation in $A - B$ of $m$ we find the rational point $(b^{1/k}, a^{1/2})$ lying on the curve $m = Y^2 - X^k$, which has genus $\lfloor \frac{k-1}{2} \rfloor \geq 2$. Therefore by uniformity, for each $m \in A-B$ there are at most $B_{\lfloor \frac{k-1}{2} \rfloor}$ pairs $(a,b)$ with $m = a-b$, implying that $r_{A-B}(m) \ll_k 1$. We now have
$$\E_l(A,B) = \sum_{m} r_{A - B}(m)^l = \sum_{m \neq 0} r_{A - B}(m)^l + r_{A - B}(0)^l \ll_{k,l} |A-B| + |A \cap B|^l$$
as needed.
\end{proof}

\subsection{Products and shifts along graphs}
In this section we prove Corollary \ref{cor:sp}. 

\begin{proof}[Proof of Corollary \ref{cor:sp}]

 Write $C=A \cdot_G B$ and $D=(A+\alpha) \cdot_G (B +\beta)$. We double count solutions to the system
\begin{align} \label{system}
    c&=ab \nonumber
    \\ d&=(a+\alpha)(b+\beta)\,\,\,\,\, \text{ such that } (a,b) \in G,\,c \in C, d \in D.
\end{align}
Let $S$ denote the number of solutions to \eqref{system} and note that $S=|E(G)|$. Indeed, each $(a,b) \in G$ determines a unique $c \in C$ and $d \in D$ which solves \eqref{system}.

On the other hand, we now aim to use Theorem \ref{thm:main} with $s=3$ to upper bound $S$. 
Since $0\notin B$, we can eliminate the variable $a$ from the system \eqref{system} and rearrange to see that a solution $(a,b,c,d)$ to $S$ yields a solution $(c,d,b)$ to the polynomial equation $F(x,y,z)=0$, where 
\[F(x,y,z)=\left(\frac{x}{\alpha}+ z\right)(z+\beta) -\frac{zy}{\alpha}.
\]
Therefore, 
\[
S \leq |Z(F) \cap C \times D \times B|.
\]
$F$ can be rearranged as
\[
F(x,y,z)= \left ( z + \frac{1}{2} \left (\frac{x-y}{\alpha} + \beta \right ) \right)^2 - \left(\frac{1}{4}\left (\frac{x-y}{\alpha} + \beta \right )^2 -\frac{\beta x}{\alpha} \right) := q(x,y,z)^2-p(x,y),
\]
so that it takes the form required for the application of Theorem \ref{thm:main}. We now perform some analysis to allow us to bound $|M_{C,p}|$. We can rearrange the above expression for $p$ and see that
\[
p(x,y)=\frac{1}{4\alpha^2}y^2-\frac{x+\alpha\beta}{2\alpha^2}y+\left(\frac{x-\alpha\beta}{2\alpha}\right)^2.
\]
For a fixed value of $x$, $p(x,y)$ is a quadratic equation with discriminant $\Delta=\frac{\beta x}{\alpha^3}$. In particular, $p(x,y)$ has repeated root if and only if either $\beta = 0$ or $x=0$, which cannot happen. With a straightforward calculation we see that the roots of $p(x,y)$ are
\[
(x+\alpha \beta)\pm 2 \sqrt{\beta \alpha x}.
\]

Now fix two values $x_1$ and $x_2$. The roots of $p(x_1,y)$ and $p(x_2,y)$ are, respectively
\[
(x_1+\alpha \beta)\pm 2 \sqrt{\beta \alpha x_1}, \quad (x_2+\alpha \beta)\pm 2 \sqrt{\beta \alpha x_2}.
\]

We see that $p(x_1,y)$ and $p(x_2,y)$ have a root in common if and only if one of the four variants of the equation $\sqrt{x_1} \pm \sqrt{x_2} = \pm2\sqrt{\alpha \beta}$ holds. We can now bound $|M_{C,p}|$. For each $c_i \in C$, there exist at most four possible choices of $c_j \in C$ such that $p(c_i,y)$ and $p(c_j,y)$ share a common root. It then follows that
\[
|M_{C,p}| \leq 4 \binom{3}{2}|C|^{2}.
\]
 Finally, since $F$ is monic in $z$, we have $L_F = 0$. We can now apply Theorem \ref{thm:main} to conclude that
\[
|E(G)|=S \ll |C||D|^{2/3} + |C|^{2/3}|D|,
\]
and thus
\[
\max \{|C|,|D| \} \gg |E(G)|^{3/5},
\]
as required.

\end{proof}

We note that the same argument as above can be used to deduce the bound
\[
\max\{|A+_G B|, |A \cdot_G B| \} \gg |E(G)|^{3/5}
\]
from \cite{ShkredovSolymosi} as a consequence of our Theorem \ref{thm:main}.

\section*{Acknowledgements}

The first listed author was supported by the grant Artin Approximation, Arc-R{\"a}ume, Aufl{\"o}sung von Singularit{\"a}ten FWF P-31336. The second, third, and fourth authors were partially supported by the Austrian Science Fund FWF Project P 30405-N32. We are grateful to Antal Balog, Tom Bloom, Christian Elsholtz, Brandon Hanson, Cosmin Pohoata, Sean Prendiville, Imre Ruzsa, Ilya Shkredov, Jozsef Solymosi, and Dmitrii Zhelezov for helpful conversations.

\Addresses

\end{document}